\documentclass[12pt]{article}

\usepackage{amsmath}
\usepackage{amsthm}
\usepackage{amssymb}
\usepackage{mathptmx}
\usepackage{concrete}
\usepackage{bbding}
\usepackage{graphicx}
\usepackage{multirow,makecell}

\title{\bf On rational functions with more than three branch points}
\author{Jijian Song$^1$ and Bin Xu$^2$}
\usepackage{fancyhdr}
\def\nd{\noindent}
\newtheorem{thm}{Theorem}[section]
\newtheorem{lem}{Lemma}[section]
\newtheorem{prop}{Proposition}[section]
\newtheorem{cor}{Corollary}[section]
\newtheorem{defi}{Definition}[section]
\newtheorem{rem}{Remark}[section]

\begin{document}

\maketitle

\nd{\small  $^{1,2}$Wu Wen-Tsun Key Laboratory of Mathematics, USTC, Chinese Academy of Sciences\\
$^{1,2}$School of Mathematical Sciences, University of Science and Technology of China.\\
No. 96 Jinzhai Road, Hefei, Anhui Province\  230026\  P. R. China.}

\nd {\small $^{1}$\Envelope smath@mail.ustc.edu.cn \quad\quad \ \  $^2$bxu@ustc.edu.cn}
\par\vskip0.5cm

\nd {\small {\bf Abstract:}  Let $\Lambda$ be a collection of partitions of a positive integer $d$ of the form
       $$(a_1,\cdots, a_p),\,(b_1,\cdots, b_q),\,(m_1+1,1,\cdots,1),\cdots, (m_l+1,1,\cdots,1),$$
where $(m_1,\cdots, m_l)$ is a partition of $p+q-2>0$. We prove that there exists a rational function on the Riemann sphere $\overline{\mathbb{C}}$ with branch data $\Lambda$ if and only if $$\max\bigl(m_1,\cdots,m_l\bigr) < \frac{d}{{\rm GCD}(a_1,\cdots, a_p,b_1,\cdots, b_q)}.$$ As an application, we give a new class of branch data which can be realized by Belyi functions on the Riemann sphere.\\

\nd{\bf Keywords.} branch data, Realizability Problem, Belyi function, Riemann's existence theorem

\nd {\bf 2010 Mathematics Subject Classification.} Primary 20B35; secondary 14H30

\section{Introduction}

Let $X$ and $Y$  be two compact connected Riemann surfaces and $f: X \to Y$ a holomorphic branched covering of degree $d$. For each point $q$ in $Y$, there associates a partition $\lambda(q)=(k_1,...,k_r)$ of $d$ such that, over a suitable neighborhood of $q$ in $Y$, $f$ is equivalent to the map
  \[\{1,...,r\}\times {\mathbb D}\to {\mathbb D},\quad (j,z)\mapsto z^{k_j},\quad {\rm where}\quad {\mathbb D}:=\{z\in {\mathbb C}:|z|<1\},\]
with $q$ corresponding to $0$ in ${\mathbb D}$. For any partition $\lambda=(k_1,k_2,\cdots, k_r)$ of $d$, we define its length ${\rm Len}(\lambda)=r$. We call the partition $\lambda$ of $d$ {\it non-trivial} if ${\rm Len}(\lambda)<d$. For the morphism $f: X \to Y$, the branch set $B_f$ consists of those points $q$ in $Y$ for which $\lambda(q)$ is nontrivial. The collection $\Lambda=\{\lambda(q):q\in B_f\}$ (with repetitions allowed) is called the {\it branch data} of $f$ and
  \[v(f):=\sum_{q\in B(f)}\,\Bigl(d-{\rm Len}\bigl(\lambda(q)\bigr)\Bigr)\]
the {\it total branching } of $f$. By the Riemann-Hurwitz theorem, we have that
  \[v(f)=2g(X)-2-d\bigl(2g(Y)-2\bigr)\]
where $g(X)$ (resp. $g(Y)$) denotes the genus of $X$ (resp. $Y$). Therefore, the total branching  $v(f)$ is an even non-negative integer.

The following problem was first proposed by Edmonds-Kulkarni-Stong \cite{EKS84} and we can trace its history to Hurwitz \cite{Hur1891}.

\nd {\bf Realizability Problem.} Given a compact connected Riemann surface $Y$ and a collection $\Lambda=\{\lambda_1,\cdots,\lambda_k\}$ of non-trivial partitions of a positive integer $d$, does there exist another compact connected Riemann surface $X$ together with a branched covering $f:X\to Y$ such that $\Lambda$ is its branch data?

If it is the case above, we call the collection $\Lambda$ is {\it realizable} or {\it realized by a branched covering}.

See the classical \cite{Bo82, EKS84, Eze78, Fr76, Ger87, Huse62, KZ95, Me84, Me90, Sin70, Thom65} and the more recent \cite{Ba01, LZ04, MSS04, OP06, Pa09, PaPe09, PaPe12, PerPe06, PerPe08, Zh06}  about this problem. Here we only review some necessary background and a small part of known results which are closely related to our discussions in the sequel.

Recall that in order to be realizable, a collection $\Lambda$ should satisfy the condition that its {\it total branching}
  \[v(\Lambda):=\sum_{j=1}^k\,\bigl(d-{\rm Len}(\lambda_j)\bigr)\]
is even. We call such a collection {\it compatible}. It is proved in \cite[Theorem 9]{Huse62} and \cite[Section 3]{EKS84} that a compatible collection is always realizable if $g(Y)>0$. Hence, we always assume that $Y$ is the Riemann sphere $\overline{\mathbb C}$ in the sequel. Namely, we are going to only consider rational functions on compact Riemann surfaces. It turns out that a compatible collection is not always realizable in this case. We call a compatible collection an {\it exception} if it is not realizable. Zheng \cite{Zh06} founded by computer all the exceptions of degree $\leq 22$. Pervova-Petronio \cite{PerPe06, PerPe08} used a variety of techniques to give some new infinite series of exceptions, and they used dessins d'enfants to make a theoretical explanation to part of the exceptions given by Zheng \cite{Zh06}. Instead of constructing exceptions, Edmonds-Kulkarni-Stong \cite{EKS84} proposed the so-called {\it prime degree conjecture}, which says that each compatible collection with prime degree is realizable and this conjecture was reduced in the same paper to the collections with only three partitions. In \cite{PaPe09, PaPe12}, Pascali-Petronio  proved some results which provided strong support to this conjecture.

Characterizing branch data of all rational functions is a very deep and difficult problem, which seems far from being accessible nowadays.  Hence, it is meaningful to find reasonably simple, sufficient conditions for realizable collections. Besides the theorems in \cite{PaPe09, PaPe12, PerPe06, PerPe08}, some of the other known results are as follows: Thom \cite{Thom65} showed that a compatible collection is realizable if one partition in it has length one.  Edmonds-Kulkarni-Stong \cite[Theorem 5.4]{EKS84} proved that a compatible collection with degree $d\not=4$ is realizable when its total branching $\geq 3(d-1)$. In addition, the exceptions with $d=4$ are precisely those with partitions $(2,2),\cdots,(2,2),(3,1)$. Moreover, Boccara \cite{Bo82} obtained a complete determination of the realizability of the collection $\Lambda$ which consists of the following three partitions of $d$:
  $$(a_1,\cdots, a_p),\,(b_1,\cdots, b_q),\,(m+1,1,\cdots,1).$$
He proved that $\Lambda$ is realizable if and only it satisfies one of the following two conditions:\\

\nd $\bullet$ $v(\Lambda)\geq 2d$ is even.

\nd $\bullet$ $v(\Lambda)=2d-2$ and $\displaystyle{m< \frac{d}{{\rm GCD}(a_1,\cdots, a_p,b_1,\cdots, b_q)}}$. Note that $m=p+q-2$ in this case.\\

Generalizing the second part of Boccara's result. We show the following

\begin{thm}[Main Theorem]
\label{thm:rat}
Let $d$ and $l$ be two positive integers. Given a collection $\Lambda$ consists of $l+2$ partitions of $d$:
  \[\Lambda=\{(a_1,\cdots, a_p),\,(b_1,\cdots, b_q),\,(m_1+1,1,\cdots,1),\cdots, (m_l+1,1,\cdots,1)\}\]
with $(m_1,\cdots, m_l)$ a partition of $p+q-2>0$. Then there exists a rational function on $\overline{\mathbb{C}}$ with $\Lambda$ as its branch data if and only if
  \[\max\,\bigl(m_1,\cdots,m_l\bigr)<\frac{d}{{\rm GCD}(a_1,\cdots, a_p,b_1,\cdots, b_q)}.\]
\end{thm}

We call a rational function on a compact Riemann surface a {\it Belyi function} if it has at most three branch points. As an application of main theorem, we have

\begin{thm}
\label{thm:Belyi}
Let $d$ and $r$ be two positive integers, and a collection $\Lambda$ consist of partitions of $d$:
  \[ \Lambda = \{(a_1,\cdots, a_p),\,(b_1,\cdots, b_q),\,(c_1+1,\cdots, c_r+1,1,\cdots, 1)\} \]
where $(c_1,\cdots, c_r)$ is a partition of $p+q-2>0$. If
  \[\max\,\bigl(c_1,\cdots,c_r\bigr)<\frac{d}{{\rm GCD}(a_1,\cdots, a_p,b_1,\cdots, b_q)}\]
then the modified collection
  $${\widetilde \Lambda}:=\{(ra_1,\cdots, ra_p),\, (rb_1,\cdots, rb_q),\, (c_1+1,\cdots, c_r+1,1,\cdots,1)\}$$
of partitions of $dr$ can be realized by a Belyi function on $\overline{\mathbb{C}}$.
\end{thm}

In the remaining of the article, we give proofs of the two theorems. And at the very end, we propose a conjecture concerning rational functions on a Riemann surface of positive genus, which is a natural generalization of our main theorem

\vspace{0.5cm}

\section{Proof of main theorem}
\label{sec:rat}
In Subsection \ref{subsec:Rie} we prove the necessary part of main theorem and observe that the sufficient part can be reduced to the case of ${\rm GCD}(a_1,\cdots,a_p, b_1,\cdots, b_q)=1$. Moreover, we recall that by the Riemann existence theorem (\cite[Theorem 2, p.49]{Don11}) the sufficient part is equivalent to
the existence of certain permutations in the symmetry group $S_d:=S_{\{1,2,\cdots,d\}}$ associated with the collection $\Lambda$.  For the completeness, we give in Subsection \ref{subsec:three} a proof in our own strategy for the case of $l=1$ of main theorem which was also proved by Boccara \cite{Bo82}.
We prove Case $l\geq 2$ of main theorem in Subsection \ref{subsec:>3}.

\subsection{Riemann existence theorem}
\label{subsec:Rie}
At first, we prove the necessary part of main theorem.

\begin{proof}
Suppose that there exists a rational function $f$ on $\overline{\mathbb C}$ realizing the branch data $\Lambda$. Using suitable M{\" o}bius transformations if necessary, we can assume that $f$ has form
\begin{equation}
\label{equ:res}
  f(z)=\frac{(z-z_1)^{a_1}\cdots (z-z_p)^{a_p}}{(z-w_1)^{b_1}\cdots (z-w_q)^{b_q}}
\end{equation}
where $z_1,\,\cdots,\,z_p,\,w_1,\,\cdots,\, w_q$ are $(p+q)$ distinct complex numbers. Let
$$k= {\rm GCD}  (a_1, \cdots,  a_p,  b_1, \cdots, b_q).$$
Then we can write $f$ as $f=F^k$ for some rational function on $\overline{\mathbb C}$. And $F$ has branch data of the form
\[\{(a_1/k,\cdots, a_p/k),\,(b_1/k,\cdots, b_q/k),\,(m_1+1,1,\cdots,1),\cdots, (m_l+1,1,\cdots,1)\}.\]
Since $F$ has degree $d/k$, $\max\bigl(m_1,\cdots,m_l\bigr) < d/k$. We are done in this part.
\end{proof}

On the other hand, reversing the order of the above argument, we see that {\it if the collection
  \[\{(a_1/k,\cdots, a_p/k),\,(b_1/k,\cdots, b_q/k),\,(m_1+1,1,\cdots,1),\cdots, (m_l+1,1,\cdots,1)\}\]
is realized by some rational function, then so is $\Lambda$}. Hence, in order to show the sufficient part, we may assume ${\rm GCD}(a_1,\cdots, b_q)=1$.

We need to prepare some notions before proving the sufficient part of main theorem.

\begin{defi}
\label{defi:res}
  {\rm Let $m$ be a non-negative integer. A vector $\alpha=(a_1,\,a_2,\,\cdots,\,a_{m+2})$ in ${\Bbb Z}^{m+2}$ is called a {\it residue vector} with $(m+2)$ components if $a_1+a_2+\cdots+a_{m+2}=0$ and $a_1 a_2\cdots a_{m+2}\not=0$. Two residue vectors $\alpha=(a_1,\cdots,a_{m+2})$ and $\beta=(b_1,\cdots,b_{m+2})$ are called {\it equivalent}, denoted by $\alpha\sim\beta$, if there is a nonzero rational number $\mu$ and a permutation $\sigma$ in the symmetry group $S_{m+2}$ such that
    \[\mu\cdot\alpha=(\mu a_1,\cdots, \mu a_{m+2})=\sigma(\beta):=\bigl(b_{\sigma(1)},\cdots, b_{\sigma(m+2)}\bigr).\]
  This is an equivalence relation in the set of residue vectors with $(m+2)$ components. The degree of the residue vector $\alpha$ is defined to be
    \[\deg\, \bigl(a_1,\cdots, a_{m+2}\bigr)=\frac{\sum_{a_j>0}\, a_j}{{\rm GCD}\bigl(a_1,\cdots, a_{m+2}\bigr)}.\]
  We call a residue vector $\alpha=(a_1,\cdots, a_{m+2})$ {\it primitive} if ${\rm GCD}(a_1,\cdots, a_{m+2})=1$. Clearly, the degree of a primitive residue vector equals the sum of all its positive components. Observe also that the logarithmic differential $d\,(\log\, f)=\frac{df}{f}$ of a rational function $f$ in \eqref{equ:res} has residues $a_1,\cdots, a_p,-b_1,\cdots,-b_q$, which form a residue vector with degree $d/{\rm GCD}\bigl(a_1,\cdots, a_p, b_1,\cdots, b_q\bigr)$}.
\end{defi}

\begin{defi}
\label{def:par}
  {\rm Let $\lambda=(\lambda_1,\,\lambda_2,\,\cdots,\, \lambda_l)$ be a partition of a positive integer $n$. The {\it weight} of $\lambda$ is defined to be}
    $${\rm wt}(\lambda) = \max (\lambda_1, \cdots, \lambda_l)$$
\end{defi}

Use the notions in main theorem and denote by $\lambda$ the partition $(m_1,\cdots, m_l)$ of $m = p+q-2>0$ and by $\alpha$ the residue vector $(a_1,\cdots, a_p,-b_1,\cdots, -b_q)$. Then the condition in the theorem can be concisely re-expressed as
  \[\deg \alpha>{\rm wt}(\lambda).\]
By the Riemann existence theorem \cite[Theorem 2, p.49]{Don11}, the sufficient part of the main theorem is equivalent to the following

\begin{thm}
\label{thm:per}
  Under the assumptions of main theorem, if $\deg \alpha > {\rm wt}(\lambda)$, then there exist $(l+2)$ permutations of $\tau_1, \tau_2, \sigma_1,\cdots, \sigma_l$ in the symmetry group $S_d=S_{\{1,2,\cdots, d\}}$ of $\{1,2,\cdots, d\}$ such that
  \begin{itemize}
    \item $\tau _1 \tau _2 \sigma_1 \cdots \sigma_l=e$, where $e$ is the unit in $S_d$ and permutations are multiplied from right to left;
    \item $\tau _1$ has the type of $a_1^1a_2^1\cdots a_p^1$, $\tau _2$ of $b_1^1b_2^1\cdots b_q^1$ and $\sigma_k$ of $(1+m_k)^1 1^{d-m_k-1}$ for all $k=1,\cdots,l$;
    \item The subgroup $\langle\,\tau _1, \tau _2, \sigma_1,\cdots,\sigma_l\rangle$ of $S_d$ acts transitively on $\{1,2,\cdots,d\}$.
  \end{itemize}
\end{thm}

The following lemma will be useful later, which follows from the Riemann existence theorem and the argument in the first three paragraphs of this subsection.

\begin{lem}
\label{lem:prim}
  For each $m\geq 0$, proving Theorem \ref{thm:per} is equivalent to proving its variant where $\alpha$ is  primitive. We call the latter {\it the primitive version} of Theorem \ref{thm:per}.
\end{lem}

We shall prove Theorem \ref{thm:per} and its primitive version simultaneously by induction on $m=p+q-2$ in the sequel of this section. The proof will be divided into two parts regarding $l=1$ or $l>1$.

Without loss of generality, we may assume the residue vector satisfies the following order assumption:

\nd \underline{\bf Order Assumption (OA)}   $a_1\leq a_2\leq \cdots \leq a_p$, $b_1\geq \cdots\geq b_q$ and
$1 \leq p\leq q$.

\subsection{Branch data with three partitions}
\label{subsec:three}

For the completeness of the manuscript, we include in this subsection the proof of the well known case where $l=1$ of Theorem \ref{thm:per} (see \cite{Bo82}), whose strategy we also use while proving in Subsection \ref{subsec:>3} the $l>1$ case of the theorem.

Here we first make a recall of the case $l=1$ of Theorem \ref{thm:per}.
\begin{prop}[Case $l=1$ of Theorem \ref{thm:per}]
\label{prop:OneZero}
  Let $\alpha = (a_1,\cdots,a_p,-b_1,\cdots,-b_q)$ be a residue vector and $\lambda = (m)$ be a partition of $m=p+q-2>0$ such that ${\rm deg}\,\alpha>{\rm wt}(\lambda)=m$. Then there exist three permutations $\tau_1, \tau_2, \sigma_1$ in $S_d$ satisfying the following three properties{\rm :}
    \begin{itemize}
       \item $\tau _1 \tau _2 \sigma_1=e$;
       \item $\tau _1$ has the type of $a_1^1a_2^1\cdots a_p^1$, $\tau _2$ of $b_1^1b_2^1\cdots b_q^1$ and $\sigma_1$ of $(1+m)^1 1^{d-m-1}$;
       \item The subgroup $\langle\,\tau _1, \tau _2, \sigma_1\rangle$ of $S_d$ acts transitively on $\{1,2,\cdots,d\}$.
    \end{itemize}
\end{prop}

\begin{lem}
\label{lem:pq}
  Let $\alpha=(a_1,\cdots,a_p,-b_1,\cdots,-b_q)$ be a residue vector with $\deg \alpha >m=p+q-2$. If $m>0$, we have $a_p>b_q$.
\end{lem}
\begin{proof}
  If $a_p<b_q$, then it follows from the order assumption (OA) that $\sum\limits_{i=1}^p a_i<\sum\limits_{j=1}^q b_j$, contradicting the definition of residue vector. If $a_p-b_q=0$, then by OA and $m=p+q-2>0$ we have $p=q \geq 2$ and $a_i=b_j$ for all $i=1,2,\cdots, p$ and $j=1,2,\cdots, q$. Since $p=\deg\,\alpha\,>m=p+q-2=2p-2=2q-2$, we obtain $p=q=1$, which contradicts $p+q>2$.
\end{proof}

To prove Proposition \ref{prop:OneZero}, we need to use the following lemma, where we propose a new concept, called {\it contraction of a residue vector}.

\begin{lem}
\label{lem:contr}
  Under the assumptions of Proposition \ref{prop:OneZero}, there exist $i_0 \in \{1,2,\cdots, p\}$ and $j_0 \in \{1,2,\cdots, q\}$ such that
    $$\widehat{\alpha}=\Bigl(a_1,\cdots, a_{i_0-1},a_{i_0}-b_{j_0},a_{i_0+1},\cdots, a_p,\ -b_1,\cdots,\widehat{(-b_{j_0})},\cdots, -b_q\Bigr)$$
  is a residue vector with $\deg\,\widehat{\alpha}\,>\, m-1$, where $a_{i_0}-b_{j_0}>0$ and the hat over term $(-b_{j_0})$ means that $(-b_{j_0})$ is removed. Note that the number of components of $\widehat{\alpha}$ is one less than that of $\alpha$. We call $\widehat{\alpha}$ a {\bf contraction} of $\alpha$.
\end{lem}

\begin{proof}
  Without loss of generality, we assume $\alpha$ is primitive. If not, we may replace it by a primitive residue vector equivalent to it.

  By OA and Lemma \ref{lem:pq}, we have $q\geq 2$ and $a_p-b_q>0$, and then we obtain another residue vector
    $$\beta_1:=\bigl(a_1,\cdots, a_{p-1},\, a_p-b_q,\ -b_1, \cdots, -b_{q-1}\bigr)$$
  with $(m-1)$ components. We divide the proof into the following three steps.

  \begin{itemize}
    \item[Step 1.] Assume $q=2$. Then $m=p+q-2=p$ and $\deg \beta_1 \geq p$, so $\beta_1$ is a contraction of $\alpha$. We assume that $q>2$ in the sequel of the proof.
    \item[Step 2.] Suppose that $\beta_1$ is primitive. Then $\deg \beta_1> (m-1)=(p+q-3)$, so it gives a contraction of $\alpha$. Indeed, if not, then by OA we have
        $$2q-2\geq p+q-2 > \deg \beta_1 =-b_q+\sum\limits_{i=1}^{p}\,a_i=-b_q+\sum\limits_{j=1}^{q}\, b_j=
        \sum\limits_{j=1}^{q-1}\, b_j.$$
         thus $b_{q-1}=b_q=1$ and $b_{q-2}\leq 2$. Moreover, if $b_{q-2}=2$, then
         $b_1=\cdots=b_{q-2}=2$. Therefore
         $$m+1=p+q-1 >1+\sum\limits_{j=1}^{q-1}\, b_j=\sum\limits_{j=1}^{q}\, b_j= \sum\limits_{i=1}^p a_i= \deg \alpha,$$
          which contradicts the assumption on the degree of $\alpha$.
                    We assume that $\beta_1$ is {\bf not} primitive in the left part of the proof.
    \item[Step 3.] If $\deg \beta_1> (m-1)$, then we are done. So without loss of generality, suppose $\deg\,\beta_1 \leq (m-1)=(p+q-3)$. Let $D>1$ be the greatest common divisor of all components of $\beta_1$. Then, by OA and the definition of degree,  we obtain that
        $$2q-2\geq p+q-2 > \deg \beta_1 =\frac{a_p-b_q}{D}+\sum\limits_{i=1}^{p-1}\,\frac{a_i}{D}=\sum\limits_{j=1}^{q-1}\, \frac{b_j}{D}.$$
        Hence $b_{q-1}=D$, $D|b_j$ for all $j=1,\cdots, q-2$, $D|a_i$ for all $i=1,\cdots,p-1$, and $D|(a_p-b_q)$. Consequently, $a_p\geq b_q+D>D=b_{q-1}$ and we obtain another residue vector
          $$\beta_2:=\bigl(a_1,\cdots, a_{p-1},\,a_p-b_{q-1},\ -b_1,\cdots, -b_{q-2},\,-b_q\bigr)$$
        with $(m-1)$ components. If $\deg \beta_2 >(m-1)$, then we are done. So suppose $\deg \beta_2 \leq (m-1)$. We divide our discussion into two cases: $\beta_2$ is primitive and otherwise.
        \begin{itemize}
          \item[(a)] First, suppose that $\beta_2$ is primitive. Since $\deg \beta_2 < m$, we find that $b_q=1$ and $b_{q-2}\leq 2$ by a same argument for $\beta_1$ in step 2. As a result $2\geq b_{q-2}\geq b_{q-1}=D>1$, and thus $b_{q-2}=b_{q-1}=D=2$. By the same argument for $\beta_1$, we deduce that $b_1=\cdots=b_{q-1}=D=2$ and $a_1,\cdots,a_{p-1}$ are even. In particular, $a_1-b_q=a_1-1>0$. Hence
                $$\widehat{\alpha}:=\bigl(a_1-b_q,a_2,\cdots,a_p,\ -b_1,\cdots,-b_{q-1}\bigr)$$
              is a primitive residue vector with $(m-1)$ components.  Moreover,  $\deg \widehat{\alpha} = 2q-2\geq p+q-2=m>m-1$ and thus $\widehat{\alpha}$ forms a contraction of $\alpha$.
          \item[(b)] Second, if $\beta_2$ is not primitive. A similar argument as in the $\beta_1$ not primitive case gives that $b_q = E$ and $b_{q-2}=E$ or $2E$, where $E$ is the greatest common divisor of the components of $\beta_2$. By OA, $E=b_q\leq b_{q-1}$.
              If $D = b_{q-1}=b_q=E$, then $E|(a_p-b_{q-1})$ implying that $E$ divides all the components of $\alpha$, this contradicts the primitive property of $\alpha$.
           Hence, $E=b_q<b_{q-1}\leq b_{q-2}$. Also since $b_{q-2}=E\ {\rm or}\ 2E$,  we have $b_{q-2}=2E$ and $1\leq b_{q-2}/b_{q-1}<2$. Recall that $b_{q-1}=D$ and $D|b_{q-2}$, so $D=b_{q-1}=b_{q-2}=2E=2b_q$ and $E|a_p$, which gives the same
           contradiction as above.

        \end{itemize}
  \end{itemize}
\end{proof}

We shall prove Proposition \ref{prop:OneZero} by induction on $m$. To this end, we need two lemmas.

\begin{lem}
  \label{lem:Transitive}
  Let $\sigma$ and $\tau$ be two permutations in $S_d$ for some positive integer $d$ such that $\sigma$ is a cycle of length greater than $1$ and $\tau$ has form $\nu_1 \nu_2 \cdots \nu_r$,  where $\nu_i$'s are mutually disjoint cycles of length $d_i$ for $i=1,2, \cdots, r$ and $d_1 + d_2 +\cdots +d_r =d$. We call $\nu_j$'s {\bf cycle factors} of $\tau$. Then the following two conditions are equivalent{\rm:}
  \begin{itemize}
    \item[(i)] The subgroup $\langle \sigma, \tau \rangle$ of $S_d$ acts transitively on the set $\{1,2,\cdots,d\}$.
    \item[(ii)] Each cycle factor $\nu_i$ of $\tau$ intersects the cycle $\sigma$ in the sense that the subset of $\{1,2,\cdots,d\}$ associated with $\nu_i$ intersects that associated with the cycle $\sigma$.
  \end{itemize}
\end{lem}

\begin{proof}
  (i) $\Rightarrow$ (ii) We prove the intersection by contradiction. Suppose that there exists a cycle factor $\nu_i$ of $\tau$ not intersecting $\sigma$. Then $\nu_i$ is a cycle factor of $\tau \sigma$. Therefore, the subset associated with $\nu_i$ forms an orbit under the action of the subgroup $\langle \tau, \sigma \rangle$ of $S_d$ on the set $\{1,2,\cdots, d\}$. Since this action is transitive, the cycle $\nu_i$ has  length $d$. Hence, $\nu_i$ must intersect $\sigma$, contradiction!

  (ii) $\Rightarrow$ (i) For each $1\leq i\leq r$, we choose a number $x_i$ lying in both $\nu_i$ and $\sigma$. Consider the action of the subgroup $\langle \tau, \sigma\rangle$ on the set $\{1,\cdots, d\}$. Then $x_1,\cdots, x_r$ lie in the same orbit of this action by assumption. Moreover, for each $1\leq i\leq r$, $x_i$ belongs to the same orbit with any other numbers in the cycle $\nu_i$, and the numbers in $\nu_1\cdots\nu_r$ exhaust all the numbers in $\{1,\cdots, d\}$. Therefore, the action has a single orbit.
\end{proof}

\begin{rem}
\label{rem:l>1}
  We consider a more general case of the part of (i)$\Rightarrow$(ii) in the above lemma by replacing $\sigma$ by $\sigma_1, \sigma_2, \cdots, \sigma_l$, all of which are cycles in $S_d$ of length greater than $1$. Suppose that the subgroup $\langle \sigma_1, \sigma_2, \cdots, \sigma_l, \tau \rangle$ of $S_d$ acts transitively on the set $\{1,2,\cdots,d\}$. Then it follows from a similar argument as in the preceding proof that each cycle factor of $\tau$ must intersect some $\sigma_k$ for $1\leq k\leq l$. However, for $l\geq 2$, the converse fails in general. For example, the subgroup of $S_4$ generated by $(12), (34)$ and $(12)(34)$ acts on the set $\{1,2,3,4\}$ with the two orbits of $\{1,2\}$ and $\{3,4\}$.
\end{rem}

\begin{lem}
\label{lem:NotEmp}
  Under the assumptions of Proposition \ref{prop:OneZero}, suppose that there exist permutations $\tau_1, \tau_2, \sigma_1$ in $S_d$ such that $\tau_1\tau_2\sigma_1=e$ and they have types of $a_1^1a_2^1\cdots a_p^1$, $b_1^1b_2^1\cdots b_q^1$ and  $(1+m)^1 1^{d-m-1}$, respectively. Then the subgroup $\langle \tau_1, \tau_2, \sigma_1\rangle$ acts transitively on the set $\{1,2,\cdots,d\}$ if and only if each cycle factor of $\tau_1$ and $\tau_2$ intersects the $(m+1)$-cycle $\sigma_1$.
\end{lem}

\begin{proof}
  Since $\tau_1\tau_2\sigma_1=e$,  the following three subgroups coincide with each other:
    $$\langle \tau_1, \sigma_1 \rangle = \langle \tau_1, \tau_2, \sigma_1 \rangle = \langle \tau_2, \sigma_1 \rangle.$$
  The result follows from Lemma \ref{lem:Transitive}.
\end{proof}

Now we give the proof of Proposition \ref{prop:OneZero}.

\begin{proof}  
  We argue by induction on $m=p+q-2\geq 0$. It holds trivially as $m=0$, which is equivalent to $p=q=1$. Assume $p+q>2$ in what follows. By Lemma \ref{lem:contr}, there exists a contraction $\widehat{\alpha}$ of $\alpha$. Without loss of generality, we assume that the contraction $\widehat{\alpha}$ has the form
     $$\widehat{\alpha}=\bigl(a_1,\cdots, a_{p-1}, a_p-b_q,\ -b_1, \cdots, -b_{q-1}\bigr).$$
  By the induction hypothesis, there exist in $S_{d-b_q}$ a permutation $\tau_2=\nu_1 \nu_2 \cdots \nu_{q-1}$ of type $b_1^1 b_2^1\cdots b_{q-1}^1$ and a cycle $\sigma_1$ of length $(p+q-2)$ such that  $\tau_2\sigma_1=\nu_1 \nu_2 \cdots \nu_{q-1} \sigma_1$ has type of $a_1^1\cdots a_{p-1}^1 (a_p-b_q)^1$ and the subgroup generated by $\tau_2$ and $\sigma_1$ acts transitively on the set $\{1,2,\cdots, d-b_p\}$, where $\nu_j$ is a cycle factor of length $b_j$ of $\tau_2$ for all $j=1, 2, \cdots, q-1$. In addition, by the induction hypothesis, $\tau_2\sigma_1$ has the form
  \begin{equation*}
    \nu_1 \nu_2 \cdots \nu_{q-1} \sigma_1=\mu_1 \mu_2 \cdots \mu_p
  \end{equation*}
  where $\mu_i$'s are mutually disjoint cycles for $1\leq i\leq p$, the length of $\mu_p$ is  $(a_p-b_q)$ and $\mu_k$ has the length $a_k$ for $1\leq k\leq p-1$.

  By Lemma \ref{lem:NotEmp}, we can choose an integer $1\leq x\leq (d-b_q)$ lying in both $\mu_p$ and $\sigma_1$. Choose in $S_d$ a cycle $\nu_q$ of length $b_q$ such that $\nu_q$ does not intersect $\nu_j$ for all $j=1,\cdots,q-1$, for example $\nu_q = (d,d-1,\cdots,d-b_q+1)$ and pick an integer $y$ in  $\nu_q$. Then $\nu_1 \nu_2 \cdots \nu_{q-1} \nu_q$ has type $b_1^1 b_2^1\cdots b_q^1$ and $\widetilde{\sigma_1}:=\sigma_1 (x,y)$ is a cycle of length $(p+q-1)$. The subgroup $\langle \nu_1 \nu_2 \cdots \nu_{q-1} \nu_q, \widetilde{\sigma_1} \rangle$ of $S_d$ acts transitively on the set $\{1,2,\cdots, d\}$ by Lemma \ref{lem:NotEmp}. We also observe that $\tilde{\mu}_p:=\nu_q\,\mu_p\, (x,y)$ is a cycle of length $a_p$ since $\nu_q$ does not intersect $\mu_p$ and $x$ and $y$ lie in $\mu_p$ and $\nu_q$, respectively. Moreover, since $\nu_q$ does not intersect $\mu_i$ for all $1\leq i\leq p-1$, $\tilde{\mu}_p$ does not intersect $\mu_i$ for all $1\leq i\leq p-1$. So we have
  \begin{align*}
     \bigl(\nu_1 \nu_2 \cdots \nu_{q-1} \nu_q\bigr)\, \widetilde{\sigma_1}&=\nu_q \bigl(\nu_1 \nu_2 \cdots \nu_{q-1} \sigma_1\bigr) (x,y)
                                                          =\nu_q \bigl(\mu_1 \mu_2 \cdots \mu_p\bigr) (x,y)\\
                                                          &=\bigl(\mu_1 \mu_2 \cdots \mu_{p-1}\bigr) \bigl(\nu_q \mu_p (x,y)\bigr)
                                                          =\mu_1 \mu_2 \cdots \mu_{p-1}\tilde{\mu}_p.
  \end{align*}
  Hence, we see that the three permutations $\bigl(\mu_1 \mu_2 \cdots \mu_{p-1}\tilde{\mu}_p\bigr)^{-1}$, $ \nu_1 \nu_2 \cdots \nu_{q-1} \nu_q$ and $\widetilde{\sigma_1}$ in $S_d$ satisfy the three properties listed in Proposition \ref{prop:OneZero}.
\end{proof}

\subsection{Branch data with more than three partitions}
\label{subsec:>3}

We prove Case $l\geq 2$ of Theorem \ref{thm:per}. At first we deal with residue vectors with components only $\pm 1$.

\begin{prop}
\label{prop:1-Res}
  Let $D$ be a positive integer. Assume $\alpha=(\underbrace{1, 1, \cdots, 1}_{D+1},\underbrace{-1,-1,\cdots, -1}_{D+1})$. Then for each partition $\lambda=(m_1,m_2,\cdots, m_l) $ of $2D$ such that ${\rm wt}(\lambda)<\deg\,\alpha=D+1$,  there exist $l$ permutations $\sigma_1,\cdots,\sigma_l$ in $S_{D+1}$ such that the following properties hold
   \begin{itemize}
     \item $\sigma_1\cdots \sigma_l=e$;
     \item $\sigma_j$'s are cycles of length $(m_j+1)$;
     \item $\langle \sigma_1,\cdots, \sigma_l \rangle$ acts transitively on the set $\{1,\cdots, D+1\}$.
   \end{itemize}
\end{prop}

\begin{proof}
  It is easy to see that $l\geq 2$. We divide the proof by considering three cases.
  \begin{itemize}
    \item[Case 1] If $l=2$, we know that $m_1=m_2=D$ since $2D=m_1+m_2$ and $m_1,\, m_2\leq D$. Then we are done by choosing
       \begin{align*}
         \sigma_1=(1,2,\cdots,D+1),\quad \sigma_2=\sigma^{-1}_1.
       \end{align*}
    \item[Case 2] If $l=3$. Since $m_1,\,m_2,\,m_3\leq D$ and $m_1+m_2+m_3=2D$, We have $m_1+m_2\geq D$. Choosing
       \begin{align*}
         \sigma_1&=(1,2,\cdots,m_1+1)\quad {\rm and}\\
         \sigma_2&=
         (1,\underbrace{m_1+1,m_1,m_1-1,\cdots,m_1+m_3-D+2}_{m_1+m_2-D},\underbrace{m_1+2,m_1+3,\cdots,D+1}_{D-m_1}),
       \end{align*}
      we obtain
      \begin{displaymath}
        \sigma_1\sigma_2= \left\{
         \begin{array}{l}
            (\underbrace{m_1+2,m_1+3,\cdots,D+1}_{D-m_1},\underbrace{2,3,\cdots,m_1+m_3-D+2}_{m_1+m_3-D+1}), \quad {\rm if } \quad m_3<D \\
            (\underbrace{m_1+2,m_1+3,\cdots,D+1}_{D-m_1},\underbrace{2,3,\cdots,m_1+1, 1}_{m_1+1}), \qquad {\rm if } \quad m_3=D
         \end{array}\right.
      \end{displaymath}
      Then the permutations $\sigma_1,\sigma_2$ and $\bigl(\sigma_1\sigma_2\bigr)^{-1}$
      satisfy the three properties.
    \item[Case 3] Suppose $l>3$. Since $m_1,\cdots, m_l\leq D$, we can choose $1 < r \leq l$ such that $m_1+m_2+\cdots+m_{r-1}\leq D$ and $m_1+m_2+\cdots+m_r>D$. \\

        Subcase 3.1 Suppose that $r<l$. Choosing
       \begin{align*}
         \sigma_1&=(1,2,\cdots,m_1+1),\\
         \sigma_2&=(m_1+1,m_1+2,\cdots,m_1+m_2+1),\\
                 &\cdots\\
         \sigma_{r-1}&=(m_1+\cdots+m_{r-2}+1,\cdots,m_1+\cdots+m_{r-1}+1),
       \end{align*}
      we obtain
       \begin{align*}
         \tau_1:=\sigma_1\sigma_2\cdots\sigma_{r-1}=(1,2,3,\cdots,m_1+\cdots+m_{r-1}+1)
       \end{align*}
      By Case 2, there exist two cycles $\tau_2$ and $\tau_3$ which have length
       $1+m_r$ and $m_{r+1}+\cdots+m_l+1<D+1$, respectively, such that $\tau_1\tau_2\tau_3=e$. As the construction of $\sigma_1,\cdots,\sigma_{r-1}$, we can find $\sigma_{r+1},\sigma_{r+2},\cdots, \sigma_l$ directly such that $\sigma_j$ has the type of $(1+m_j)^11^{D-m_j}$ for $r+1\leq j \leq l$ and $\sigma_{r+1}\cdots\sigma_l=\tau_3$. Therefore the $l$ cycles of  $\sigma_1,\, \cdots, \sigma_{r-1},\,\sigma_r:=\tau_2,\,\sigma_{r+1},\cdots, \sigma_l$ satisfy the three properties.\\

       Subcase 3.2  Suppose $r=l>3$. Since $m_1+\cdots+m_l=2D$ and $\max\,\bigl(m_1,\cdots, m_l\bigr)\leq D$, we have $m_1+\cdots+m_{l-1}=m_l=D$. Then the problem can be reduced to Case 1 by a similar argument as above.
  \end{itemize}
\end{proof}

To complete the proof of Theorem \ref{thm:per}, we need the following lemma and its two corollaries.

\begin{lem}
  \label{lem:OneOrbit}
  Let $\Gamma$ be a subgroup of $S_d=S_{\{1,2,\cdots, d\}}$ for some integer $d>1$ and $\theta \in S_d$ be a cycle of length greater than $1$. Assume that the subgroup $G$ generated by $\Gamma$ and $\theta$ acts transitively on the set $\{1,2,\cdots,d\}$. Then, for each number $1 \leq x \leq d$ not contained in $\theta$, the $\Gamma$-orbit $\Gamma x$ of $x$ intersects $\theta$.
\end{lem}

\begin{proof}
  We argue by contradiction. Suppose that the orbit $\Gamma x$ does not intersect $\theta$. Take an arbitrary permutation $\xi$ in $G$. We can express it as
    $$\xi = \pi_1\pi_2 \cdots \pi_s$$
  where either $\pi_i=\theta$ or $\pi_i \in \Gamma$. Let $\xi'$ be the permutation obtained from the product $\pi_1\pi_2 \cdots \pi_s$ by removing all those $\pi_i$'s satisfying $\pi_i=\theta$. Then $\xi' \in \Gamma$. Since each number not contained in $\theta$ is a fixed point of $\theta$, by the hypothesis of the contradiction argument, we find that $\xi(x)=\xi'(x)$ is not contained in $\theta$. Since $\xi \in G$ has been chosen arbitrarily, the orbit $Gx$ does not intersect $\theta$, which contradicts that $G$ acts transitively on $\{1,2,\cdots, d\}$.
\end{proof}

As an application of the above lemma, we have

\begin{cor}
  \label{cor:SubgroupTran-A}
  Let $\gamma_1, \gamma_2, \cdots, \gamma_l, \theta$ be $(l+1)$ permutations in $S_{d-1}$ for some integer $d>2$ and $\theta = (x_1, x_2, \cdots, x_n)$ be a cycle in $S_{d-1}$ of length $n>1$. Suppose that the subgroup $\langle \gamma_1, \gamma_2, \cdots, \gamma_l, \theta \rangle$ acts transitively on the set $\{1, 2, \cdots, d-1 \}$.  Then so does the subgroup $\langle \gamma_1, \gamma_2, \cdots, \gamma_l, \widetilde{\theta} \rangle$ on $\{1, 2, \cdots, d\}$, where $\widetilde{\theta}:=(x_1, \cdots, x_n, d)$ is a cycle of length $(n+1)$ in $S_d$.
\end{cor}
\begin{proof}
  By Lemma \ref{lem:OneOrbit}, we can see that for each number $1 \leq x \leq d-1$ not contained in $\theta$, there exists $\gamma \in \Gamma:=\langle \gamma_1, \cdots, \gamma_l \rangle$ such that $\gamma(x)$ is contained in $\theta$. Hence, the action of $\langle \gamma_1, \gamma_2, \cdots, \gamma_l, \widetilde{\theta} \rangle$ on $\{1,2\cdots,d\}$ has only one orbit.
\end{proof}

Similarly, we obtain

\begin{cor}
  \label{cor:SubgroupTran-B}
  Let $\gamma_1, \gamma_2, \cdots, \gamma_l, \theta$ be permutations in $S_d$ for some integer $d>2$ and $\theta = (x_1, x_2, \cdots, x_n)$ be a cycle of length $1<n<d$. Suppose that the subgroup $\langle \gamma_1, \gamma_2, \cdots, \gamma_l, \theta \rangle$ acts transitively on the set $\{1, 2, \cdots, d \}$.  Then so does the subgroup $\langle \gamma_1, \gamma_2, \cdots, \gamma_l, \widetilde{\theta} \rangle$ of $S_d$ on the set $\{1, 2, \cdots, d\}$, where $\widetilde{\theta}=(x_1, \cdots, x_n, y) \in S_d$ is a cycle of length $(n+1)$ with $y\in \{1,2,\cdots,d\}\backslash \{x_1,\cdots, x_n\}$.
\end{cor}

Now we arrive at proving the case $l\geq 2$ of Theorem \ref{thm:per}.

\begin{proof}
   By Lemma \ref{lem:prim} we could also assume that the residue vector $\alpha=(a_1,\cdots, a_p, -b_1,\cdots, -b_q)$ is primitive so that $\deg\,\alpha=d=a_1+\cdots+a_p=b_1+\cdots+b_q$.
   \begin{itemize}
     \item[\textbf{Part I}] Suppose $d \geq m+1= p+q-1$.
       By Proposition \ref{prop:OneZero} there exist $\tau_1, \tau_2$, $\sigma$ such that\\
        (1) $\tau _1 \tau _2 \sigma=e$;\\
        (2) $\tau _1$ has type of $a_1^1a_2^1\cdots a_p^1$, $\tau _2$  of $b_1^1b_2^1\cdots b_q^1$, $\sigma$ of $(1+m)^1 1^{d-m-1}$;\\
        (3) the subgroup $\langle \tau _1, \tau _2, \sigma\rangle$ acts transitively on $\{1,2,\cdots,d\}$.\\
       Assume that $\sigma=(1, 2, \cdots, m+1)$ for simplicity of notion. We are done by choosing
       \begin{align*}
         \sigma_1&= (1,2, \cdots, m_1+1),\\
         \sigma_2&= (m_1+1,m_1+2, \cdots, m_1+m_2+1),\\
         \cdots  &\cdots \\
         \sigma_l&=(m_1+\cdots+m_{l-1}+1,m_1+\cdots+m_{l-1}+2,\cdots, m_1+\cdots+m_l+1),
       \end{align*}

     \item[\textbf{Part II}] Suppose $d={\rm deg}\,\alpha \leq m = p+q-2$.  We first reduce the problem to the two cases that $l=2$ and $l=3$, then we prove these two cases by using the contraction argument and the induction argument similar as the proof of Proposition \ref{prop:OneZero}. The details given as follows form the left part of this section.

         By OA, we have that $d = \sum_{j=1}^q b_j \geq q \geq \frac{p+q}{2}$. Since $d={\rm deg}\,\alpha> {\rm wt}(\lambda)=\max\,\bigl(m_1,\cdots, m_l)$,  the partition $\lambda$ of $m=(p+q-2)$ has at least two components, i.e. $l>1$.

         At first we show that

         \nd {\sc Claim 1}: {\it the problem can be reduced to the two cases where the partitions of $m$ have two and three components, respectively.}

         \nd {\sc Proof of Claim 1}

         Suppose that Part II holds for each partition of $m$ which has  three components and has weight less than $d$. Then we shall prove that so does Part II for each partition $(m_1,\cdots, m_l)$ of $m$ such that $l>3$ and its weight is less than $d$. To this end, since $m_1,\cdots, m_l<d$, we can choose $1<r\leq l$ such that
           \[m_1+\cdots+m_{r-1}<d\quad {\rm and}\quad m_1+\cdots+m_r\geq d.\]
         We shall define a new partition, called $\lambda'$, with three components as following.

         \nd $\bullet$ Suppose $r<l$. Then we consider the partition $\lambda':= (m_1' , m_2' , m_3') $ of $m$, where $m_1',m_2',m_3'$ are defined by
           $$m_1':=m_1+ \cdots + m_{r-1}<d,\quad m_2':=m_r<d,\quad m_3':=m_{r+1}+ \cdots +m_l.$$
         Moreover, since $d\geq \frac{p+q}{2}$, we observe that
         \begin{align*}
            m_3'=p+q-2-(m_1'+m_2') < (p+q-d)-1 \leq \frac{p+q}{2}-1 \leq d-1.
         \end{align*}
         \nd $\bullet$ Suppose $r=l>3$. Then we choose the partition $\lambda':= \bigl(m_1+\cdots+m_{l-2}, m_{l-1} , m_l\bigr)$ of $m$, which has weight less than $d$.

         \nd  Since the partition $\lambda'$ has weight less than $d$, and we have assumed the validity of Case $l=3$ for $\alpha$ and $\lambda'$, we can find in $S_d$ the following five permutations $\tau_1, \tau_2, \sigma_1', \sigma_2', \sigma_3'$ which satisfy the three properties. By a similar construction as Case 3 in the proof of Proposition \ref{prop:1-Res}, we know that the proposition holds for the partition $(m_1,\cdots,m_l)$. Therefore, we have justified the claim.

         We always assume $l=2$ or $3$ in the left part of the proof.

         Recall that OA states that $1\leq p \leq q, a_1\leq a_2\leq \cdots \leq a_p,\, b_1\geq b_2\geq \cdots \geq b_q$. Without loss of generality, we further assume $m_1\leq m_2 \leq \cdots \leq m_l$ for the partition $\lambda$.  We call these two assumptions OA in the sequel by an abuse of notation.

         By OA, we can see that if $a_p=1$, then $a_1=a_2= \cdots =a_p = b_1=b_2=\cdots =b_q=1$ and we are done by Proposition \ref{prop:1-Res}. We may assume that $a_p > 1$ in the left part of the proof. Since $d \leq m$, we have $\sum\limits_{j=1}^q b_j \leq p+q-2 \leq 2q-2$. Then, we have $b_{q-1}=b_q=1$. Since $m_1\leq \cdots \leq m_l$ and $a_p>1$, $\deg \alpha > {\rm wt}(\lambda)$ and $l=2$ or $3$, we observe that $m_l$ is always greater than $1$ except when $\alpha=(2,1,-1,-1,-1)$ and $\lambda=(1,1,1)$, for which Part II holds trivially. We may assume that $m_l>1$ in the left part of the proof. In order to do induction on $m$, we choose the partition
           $$\lambda_1:=\bigl(m_1,\cdots,m_{l-1},(m_l -1)\bigr)$$  of $(m-1)$ and the residue vector $$\widehat{\alpha} =\bigl(a_1, a_2, \cdots, a_{p-1}, a_p-b_q, -b_1, -b_2, \cdots, -b_{q-1}\bigr)$$
         with $(p+q-1)=(m+1)$ components. Then, since $b_{q-1}=b_q=1$,  $\widehat{\alpha}$ is primitive and
           $${\rm deg}\,\widehat{\alpha}= -1+\deg\,\alpha.$$
         Then we make the following

         \nd {\sc Claim 2}:  ${\rm wt}(\lambda_1)\, < \,\deg\widehat{\alpha}$. Hence we may think of $\bigl(\widehat{\alpha},\,\lambda_1\bigr)$ as a contraction of $\bigl(\alpha,\,\lambda\bigr)$ and we shall use the former to do induction argument.

         \nd {\sc Proof of Claim 2}

         \begin{itemize}
           \item[(i)] If $m_1 \leq \cdots \leq m_{l-1}=m_l$, then ${\rm wt}(\lambda_1)=m_{l-1}=m_l$ and $2m_l\leq m =p+q-2 \leq 2q-2$, i.e. $m_l \leq q-1$. If $\deg\, \widehat{\alpha}>q-1$, then we are done. Assume $\deg \widehat{\alpha} = q-1$, which implies that $b_1=\cdots = b_q=1$ and $p < q$ since $a_p >1$. Then $2m_l\leq m =p+q-2 < 2q-2$ and $m_l < q-1$. Hence $\deg\,\widehat{\alpha}=q-1>m_l={\rm wt}(\lambda_1)$.
           \item[(ii)]  If $m_{l-1} < m_l$, then ${\rm wt}(\lambda_1)=m_l-1={\rm wt}(\lambda)-1 <\deg\,\alpha-1 =\deg\,\widehat{\alpha}$.
           The claim is proved.
         \end{itemize}

     \nd Then in the left part of the proof we use the induction on $m$ to prove that Part II of Theorem \ref{thm:per} holds provided that $l$ equals either $2$ or $3$. We observe that the initial case of $m=2$ holds trivially, where $\alpha=(1,1,-1,-1)$ and $\lambda=(1,1)$.
     \begin{itemize}
       \item Suppose $l=2$. We recall our setting as follows. Take a primitive residue vector $$\alpha=(a_1,\cdots, a_p,-b_1,\cdots,-b_q)$$ and a partition $\lambda=(m_1,\,m_2)$ of $m=p+q-2 \geq 3$  such that $d=\deg\, \alpha \leq m$, and $\alpha$ and $\lambda$ satisfy OA. Then we have $2 \leq m_2={\rm wt}(\lambda)<d$ and $b_{q-1}=b_q=1$.

           Then, by claim 2, we could take another primitive residue vector $$\widehat{\alpha}=(a_1, a_2, \cdots, a_{p-1}, a_p-b_q, -b_1, -b_2, \cdots, -b_{q-1})$$ and another partition $\lambda_1=(m_1,\,m_2-1)$ such that ${\rm wt}(\lambda_1)<\deg\,\widehat{\alpha}=d-1$. By the induction hypothesis, there exist in $S_{d-1}=S_{\{1,2,\cdots, d-1\}}$ a permutation of type $b_1^1\cdots b_{q-1}^1$, called $\tau_2$,  and two cycles $\sigma_1, \sigma_2$ of length $(1+m_1),\, m_2$, respectively, such that the subgroup $\langle \sigma_1, \sigma_2, \tau_2 \rangle$ of $S_{d-1}$ acts transitively on $\{1,2,\cdots, d-1\}$ and the permutation
           \begin{align*}
             \tau_1:=\tau_2\sigma_1\sigma_2
           \end{align*}
           has the type of $a_1^1\cdots a_{p-1}^1 (a_p-b_q)^1$. We re-express $\tau_1$ by $\tau_1=\mu_1\cdots \mu_p$ such that $\mu_j$'s are its cycle factors and $\mu_j$ has length $a_j$ for $1\leq j<p$ and $\mu_p$ has length $(a_p-b_q)$. Since $b_q=1$, we can think of $\tau_2$ as a permutation in $S_d$ with the type of $b_1^1\cdots b_{q-1}^1b_q^1$. By Remark \ref{rem:l>1},  the cycle factor $\mu_p$ intersects either $\sigma_1$ or $\sigma_2$. We shall divide the left part of the proof of Case $l=2$ into the following two steps.
         \begin{itemize}
           \item[Step 2.1] Suppose that $\sigma_2$ intersects $\mu_p$. Then we pick a number $x$ in both $\sigma_2$ and $\mu_p$ and define $\widetilde{\sigma}_2:=\sigma_2 (x,d)$ and $\widetilde{\mu}_p:=\mu_p (x,d)$. Then, since $b_q=1$,  $\widetilde{ \sigma }_2$ and $\widetilde{\mu}_p$ are cycles in $S_d$ of length $(1+m_2)$ and $a_p$, respectively. Moreover, we have
              \begin{align*}
                \tau_2 \sigma_1 \widetilde{ \sigma }_2 = \mu_1 \mu_2 \cdots \mu_{p-1} \widetilde{\mu}_p.
              \end{align*}
              Since $\langle \tau_2, \sigma_1, \sigma_2 \rangle$ acts transitively on the set $\{1,2,\cdots, d-1\}$,  the action of $\langle \tau_2, \sigma_1, \widetilde{\sigma}_2 \rangle$ on the set $\{1,2,\cdots,d\}$ has a single orbit by Corollary \ref{cor:SubgroupTran-A}.
           Therefore, the following four permutations $\bigl(\mu_1 \mu_2 \cdots \mu_{p-1} \widetilde{\mu}_p\bigr)^{-1},\, \tau_2,\, \sigma_1,\, \widetilde{ \sigma }_2$  satisfy the three properties.
          \item[Step 2.2] Suppose that $\sigma_2$ does not intersect $\mu_p$. Then $\sigma_1$ intersects $\mu_p$. Choose a number $x$ in both $\sigma_1$ and $\mu_p$. Then
              \begin{align*}
                \tau_2 \sigma_1 \sigma_2 (x,d) &=\mu_1 \cdots \mu_p (x,d)\\
                                              &=\mu_1 \cdots \mu_{p-1} \widetilde{\mu}_p \quad {\rm where}\quad
                                              \widetilde{\mu}_p:=\mu_p\,(x,d) \\
                                              &=: \widetilde{\tau}_1
              \end{align*}
              where $\widetilde{\tau}_1$ has the type of $a_1^1\cdots a_p^1$.
              Meanwhile, since $\sigma_2$ does not intersect $(x,d)$,  we have
              $$\widetilde{\tau}_1=\tau_2 \sigma_1 \sigma_2 (x,d)=\tau_2 \sigma_1 (x,d) \sigma_2 =: \tau_2 \sigma_1' \sigma_2 ,$$ where $\sigma_1':=\sigma_1\,(x,d)$ is a cycle of length $(m_1+2)$. Observe that $\sigma_1'$ intersects $\sigma_2$. For, otherwise, $\sigma_1$ does not intersect $\sigma_2$ and there are $(m_1+m_2+1)$ different numbers appearing in both $\sigma_1$ and $\sigma_2$, lying in $S_{d-1}$. Hence, we have
              $$(d-1)\geq 1+m_1+m_2=1+m\geq 1+d,$$
              contradiction!

              Take  the smallest positive integer $s$ such that $y:=\bigl(\sigma_1'\bigr)^{-s}(x)$ is contained in $\sigma_2$. Since $x$ is not contained in $\sigma_2$, by the minimal property of $s$, $\bigl(\sigma_1'\bigr)(y)$ is not contained in $\sigma_2$. We can rewrite $\sigma_1'$ as
              \begin{align*}
                \sigma_1'=(x_1,x_2,\cdots,x_{m_1},y,\sigma_1'(y))=\widetilde{\sigma}_1
                \bigl(y,\sigma_1'(y)\bigr)\quad {\rm with}\quad \widetilde{\sigma}_1:=(x_1,\cdots, x_{m_1},y).
              \end{align*}
              Then we have
              \begin{align*}
                \sigma_1'\sigma_2=
                \widetilde{\sigma}_1 \bigl(y,\sigma_1'(y)\bigr)
                \sigma_2
                =\widetilde{\sigma}_1 \widetilde{\sigma}_2\quad {\rm where}\quad
                \widetilde{\sigma}_2:= \bigl(y,\sigma_1'(y)\bigr)\sigma_2,
              \end{align*}
              and
              \begin{align*}
                \widetilde{\tau}_1=\tau_2\sigma_1'\sigma_2=
                \tau_2 \widetilde{\sigma}_1 \widetilde{\sigma}_2,
              \end{align*}
              where $\widetilde{\tau}_1$ and $\tau_2$ have types of $a_1^1 \cdots a_p^1$ and $b_1^1\cdots b_{q-1}^1 b_q^1$, respectively, and the two cycles $\widetilde{\sigma}_1$ and $\widetilde{\sigma}_2$ have lengths of $1+m_1$ and $1+m_2$, respectively. Since $\langle \tau_2, \widetilde{\tau}_1, \sigma_2 \rangle = \langle \tau_2, \sigma_1', \sigma_2 \rangle$ acts transitively on the set $\{1,2,\cdots, d\}$ by Corollary \ref{cor:SubgroupTran-A}, so does $\langle \tau_2, \widetilde{\tau}_1, \widetilde{\sigma}_2 \rangle$ on the same set by Corollary \ref{cor:SubgroupTran-B}. Therefore, the following four permutations $\bigl(\widetilde{\tau}_1 \bigr)^{-1},\,\tau_2,\,\widetilde{\sigma}_1,\, \widetilde{\sigma}_2$  satisfy the three properties.
        \end{itemize}
       \item Suppose $l=3$. We may further assume that $m_1+m_2\geq d$. Otherwise, replacing $\lambda$ by $\lambda'=(m_1+m_2,\, m_3)$ and using the similar reduction argument as above,  we can reduce the problem to the known case of $l=2$.

           By the induction hypothesis, there exist in $S_{d-1}$ permutation $\tau_2$ of type $b_1^1\cdots b_{q-1}^1$ and three cycles $\sigma_1, \sigma_2, \sigma_3$ of length $1+m_1, 1+m_2, m_3$, respectively, such that the subgroup $\langle \sigma_1, \sigma_2, \sigma_3, \tau_2 \rangle$ acts transitively on the set $\{1,2,\cdots, d-1\}$ and
           \begin{align*}
             \tau_2 \sigma_1 \sigma_2 \sigma_3= \tau_1=\mu_1\cdots\mu_p,
           \end{align*}
           where $\tau_1=\mu_1\cdots\mu_p$ has type $a_1^1\cdots a_{p-1}^1(a_p-b_q)^1$, $\mu_j$'s are the cycle factors of $\tau_1$ and $\mu_j$ has length $a_j$ for $1\leq j<p$ and $\mu_p$ has length $(a_p-b_q)=(a_p-1)$. Since $m_1+m_2\geq d$, by OA,  any two of the three cycles $\sigma_1,\sigma_2,\sigma_3$ in $S_{d-1}$ intersect since $1+m_i+m_j>d>d-1$ for $1\leq i\not=j\leq 3$. We divide the left part of the proof into the following three steps.
           \begin{itemize}
             \item[Step 3.1] If $\mu_p$ intersects $\sigma_3$, we are done by a similar argument as in step 2.1.
             \item[Step 3.2] If $\mu_p$ intersects $\sigma_2$ but does not intersect $\sigma_3$, the proof goes through as Step 2.2 since $\sigma_2$ intersects $\sigma_3$.
             \item[Step 3.3] Suppose that neither $\sigma_2$ nor $\sigma_3$ intersects $\mu_p$. Then, by Remark \ref{rem:l>1},  $\mu_p$ intersects $\sigma_1$ since $\langle \tau_1,\,\sigma_1,\,\sigma_2,\,\sigma_3\rangle$ acts transitively on the set $\{1,2,\cdots, d-1\}$. Choosing a number $x$ in both $\mu_p$ and $\sigma_1$ and denoting $\widetilde{\mu}_p:=\mu_p(x,\,d)$ and $\sigma_1':=\sigma_1(x,\, d)$, we obtain
                 \begin{align*}
                   \tau_2 \sigma_1'\sigma_2 \sigma_3 =\tau_2 \sigma_1\sigma_2 \sigma_3 (x,\, d)= \mu_1\cdots \mu_{p-1} \widetilde{\mu}_p =:\widetilde{\tau}_1
                 \end{align*}
                 where $\sigma_1'$ is a cycle of length $(m_1+2)$, $\widetilde{\mu}_p$ is a cycle factor of $\widetilde{\tau}_1$, and $\widetilde{\tau}_1$ has the type of $a_1^1\cdots a_{p-1}^1 a_p^1$.
                \begin{itemize}
                  \item[Step 3.3.A] Suppose that $\sigma_2$ contains a number which is not contained in $\sigma_3$. Then 
                    we could find in $S_d$ the three cycles of $\widetilde{\sigma}_1, \widetilde{\sigma}_2, \widetilde{\sigma}_3$ with length $1+m_1, 1+m_2, 1+m_3$, respectively, such that
                      \[\tau_2 \widetilde{\sigma}_1\widetilde{\sigma}_2 \widetilde{\sigma}_3 = \widetilde{\tau}_1\]
                      and  $\langle \tau_2, \widetilde{\tau}_1, \widetilde{\sigma}_1, \widetilde{\sigma}_3 \rangle$
                    acts transitively on $\{1, 2, \cdots, d\}$.  Indeed, rewrite $\sigma_1'$ as
                    \begin{align*}
                      \sigma_1'=\widetilde{\sigma}_1 (b_1,z_1)
                    \end{align*}
                    where $b_1$ lies in the intersection of $\sigma_1$ and $\sigma_2$, the number $z_1$ is contained in $\sigma_1'$ but not in  $\sigma_2$, and $\widetilde{\sigma}_1$ is a cycle of length $(1+m_1)$. Choose $\sigma_2':=(b_1,z_1)\sigma_2$ and rewrite it as
                    \begin{align*}
                      \sigma_2' = \widetilde{\sigma}_2 (b_2,z_2),
                    \end{align*}
                    where $b_2$ lies in both $\sigma_2$ and $\sigma_3$, the number  $z_2$ is contained in $\sigma_2$ but not in $\sigma_3$, $\widetilde{\sigma}_2$ is a $(1+m_2)$-cycle. Hence $\widetilde{\sigma}_3:= (b_2,z_2)\sigma_3$ is a $(1+m_3)$-cycle and the equality $\tau_2 \widetilde{\sigma}_1\widetilde{\sigma}_2 \widetilde{\sigma}_3 = \widetilde{\tau}_1$ holds. Then all the following subgroups 
                    \begin{align*}
                       \langle \tau_2, \sigma_1', \sigma_2, \sigma_3 \rangle = \langle \tau_2, \widetilde{\tau}_1, \sigma_2, \sigma_3 \rangle, \quad
                       \langle \tau_2, \widetilde{\tau}_1, \sigma_2', \sigma_3 \rangle = \langle \tau_2, \widetilde{\tau}_1, \widetilde{\sigma}_1, \sigma_3 \rangle
                    \end{align*}
                    and $\langle \tau_2, \widetilde{\tau}_1, \widetilde{\sigma}_1, \widetilde{\sigma}_3 \rangle$ of $S_d$
                    act transitively on $\{1, 2, \cdots, d\}$ by Corollaries \ref{cor:SubgroupTran-A} and \ref{cor:SubgroupTran-B}.
                    Therefore, the following permutations
                    $$\bigl(\widetilde{\tau}_1\bigr)^{-1},\,\tau_2,\, \widetilde{\sigma}_1,\, \widetilde{\sigma}_2,\, \widetilde{\sigma}_3$$
                     satisfy the three properties.

                 \item[Step 3.3.B]Suppose that every number in the cycle $\sigma_2$ is contained in $\sigma_3$. Rewrite $\sigma_1'$ as
                 \begin{align*}
                   \sigma_1' = \widetilde{\sigma}_1 (b,z)
                 \end{align*}
                 where $\widetilde{\sigma}_1$ is a $(1+m_1)$-cycle, $b$ lies in both $\sigma_1'$ and $\sigma_3$, $z$ is not contained in $\sigma_3$ and then it is not contained in $\sigma_2$, either.

                \nd $\bullet$ Suppose that $b$ is not contained in $\sigma_2$. Then
                \begin{align*}
                  \tau_2 \sigma_1' \sigma_2 \sigma_3 &= \tau_2 \widetilde{\sigma}_1 (b,\,z)\sigma_2 \sigma_3\\
                                                     &= \tau_2 \widetilde{\sigma}_1 \sigma_2 ((b,z)\sigma_3)\\
                                                     &= \tau_2 \widetilde{\sigma}_1 \sigma_2 \widetilde{\sigma}_3,
                \end{align*}
                where $\widetilde{\sigma}_3:=(b,z)\sigma_3$ is a $(1+m_3)$-cylce.
                By Corollary \ref{cor:SubgroupTran-B}, the subgroup $\langle \tau_2, \widetilde{\tau}_1, \sigma_2, \widetilde{\sigma}_3 \rangle$ acts transitively on the set $\{1,2,\cdots, d\}$ since the subgroup $\langle \tau_2, \widetilde{\tau}_1, \sigma_2, \sigma_3 \rangle =  \langle \tau_2, \sigma_1', \sigma_2, \sigma_3 \rangle$ has the same property. Therefore, the following five permutations $\bigl(\widetilde{\tau}_1\bigr)^{-1},\,\tau_2,\, \widetilde{\sigma}_1,\, \sigma_2,\, \widetilde{\sigma}_3$ satisfy the three properties.

                \nd $\bullet$ Suppose that $b$ is contained in $\sigma_2$. Then we have
                  $$\tau_2 \sigma_1' \sigma_2 \sigma_3 = \tau_2 \widetilde{\sigma}_1 ((b,\,z)\sigma_2) \sigma_3 = \tau_2 \widetilde{\sigma}_1 \sigma_2' \sigma_3\quad {\rm where}\quad \sigma_2':=(b,z)\sigma_2.$$
                Then applying a similar argument in Step 2.2 to $\sigma_2'$ and $\sigma_3$, we obtain the two cycles $\widetilde{\sigma}_2$ and $\widetilde{\sigma}_3$ with length of $(1+m_2),(1+m_3)$, respectively, such that
                  $$\sigma_2'\sigma_3=\widetilde{\sigma}_2 \widetilde{\sigma}_3\quad {\rm and}\quad \tau_2 \widetilde{\sigma}_1\widetilde{\sigma}_2 \widetilde{\sigma}_3 = \widetilde{\tau}_1.$$
                Moreover, by Corollaries \ref{cor:SubgroupTran-A} and \ref{cor:SubgroupTran-B},  the action of $\langle \tau_2, \widetilde{\tau}_1, \widetilde{\sigma}_1, \widetilde{\sigma}_3 \rangle$ on the set $\{1,2,\cdots, d\}$ is transitive. Therefore, the following five permutations $\bigl(\widetilde{\tau}_1\bigr)^{-1},\,\tau_2,\, \widetilde{\sigma}_1,\, \widetilde{\sigma}_2,\, \widetilde{\sigma}_3$ satisfy the three properties.
              \end{itemize}
          \end{itemize}
       \end{itemize}
    \end{itemize}
\end{proof}

\section{Proof of Theorem \ref{thm:Belyi}}
\label{subsec:Belyi}
We prove Theorem \ref{thm:Belyi} in this subsection.

\begin{proof}
  Consider the collection
    \[\Lambda^*=\{(a_1,\cdots, a_p),\,(b_1,\cdots, b_q),\,(c_1+1,1,\cdots, 1),\cdots,(c_r+1,1,\cdots,1)\}\]
  of degree $d$ and with total branching $2d-2$. By Theorem \ref{thm:per} and the Riemann existence theorem, there exists a rational function $f$ of degree $d$ on the Riemann sphere such that
  \begin{itemize}
    \item[(i)] $f$ branches over $0$, $\infty$,  $\zeta$, $\cdots$, $\zeta^r$, where $\zeta=\exp\,\bigl(2\pi\sqrt{-1}/r\bigr)$.
    \item[(ii)] The partitions of the above branch points coincide with $(a_1,\cdots, a_p),\,(b_1,\cdots, b_q),\,(c_1+1,1,\cdots, 1),\cdots,(c_r+1,1,\cdots,1)$, respectively.
  \end{itemize}
  Then $f^r$ is the Belyi function as desired.
\end{proof}


\section{A conjecture}

In order to generalize the first part of Boccara's result, we make the following

\nd {\bf Conjecture} Let $d,\,g$ and $l$ be three positive integers. Suppose that the collection $\Lambda$ consists of $l+2$ partitions of $d$ and has form
  \[\Lambda=\{(a_1,\cdots, a_p),\,(b_1,\cdots, b_q),\,(m_1+1,1,\cdots,1),\cdots, (m_l+1,1,\cdots,1)\},\]
where $(m_1,\cdots, m_l)$ is a partition of $p+q-2+2g$. Then there exists a rational function on some compact Riemann surface of genus $g$ with branch data $\Lambda$ if and only if
  \[\max\,\bigl(m_1,\cdots,m_l\bigr)<d.\]

\nd {\bf Acknowledgments.}  The authors would like to express his deep gratitude to Professor Qing Chen, Professor Mao Sheng and Mr Bo Li for their valuable discussions during the course of this work. The second author is indebted to Mr Weibo Fu for informing him  \cite{Bo82, EKS84, Zh06}.

The first author is supported in part by the National Natural Science Foundation of China (grant no.
11426236). The second author is supported in part by the National Natural Science Foundation of China (grants no. 11571330 and no. 11271343) and
the Fundamental Research Funds for the Central
Universities (grant no. WK3470000003).}




\end{document}